\documentclass{amsart}
\usepackage{amssymb}
\usepackage{graphicx}
\usepackage{epstopdf}
\usepackage{pdfsync}

\numberwithin{equation}{section}

\theoremstyle{plain} 
\newtheorem{thm}[equation]{Theorem}
\newtheorem{cor}[equation]{Corollary}
\newtheorem{lem}[equation]{Lemma}
\newtheorem{prop}[equation]{Proposition}
\newtheorem{claim}[equation]{Claim}

\theoremstyle{definition}

\theoremstyle{remark}

\graphicspath{{PSF_figures/}}
\title[On nonhyperbolicity of P/P and P/SF knots in $S^3$] {On nonhyperbolicity of P/P and P/SF knots in $S^3$}

\author{Sungmo Kang}

\email{skang4450@chonnam.ac.kr}

\begin{document}







\begin{abstract}
In \cite{B90} or an available version \cite{B18}, Berge constructed twelve families of primitive/primitve(or simply P/P) knots, which are referred to as the Berge knots. It is proved in \cite{B08} or independently in \cite{G13} that all P/P knots are the Berge knots, which gives the complete classification of P/P knots. However, the hyperbolicity of the Berge knots was undetermined in \cite{B90}. As a natural generalization of P/P knots, Dean introduced primitive/Seifert(or simply P/SF) knots in $S^3$. The classification of hyperbolic P/SF knots has been performed and the complete list of hyperbolic P/SF knots is given in \cite{BK20}.

In this paper, we provide necessary, sufficient, or equivalent conditions for P/P or P/SF knots being nonhyperbolic, and their applications on some P/SF knots. The results of this paper will be used in proving the hyperbolicity of P/P and P/SF knots in \cite{K20a} and \cite{K20b} completing the classification of hyperbolic P/P and P/SF knots in \cite{B90} and \cite{BK20} respectively.
\end{abstract}

\maketitle


\section{Introduction}\label{Introduction and main result}

For a genus two handlebody $V$ and an essential simple closed curve $\alpha$ in $\partial V$, $V[\alpha]$ will denote the 3-manifold obtained by
adding a 2-handle to $V$ along $\alpha$. An essential simple closed curve $\alpha$ in $\partial V$ is \textit{primitive} in $V$ if $V[\alpha]$ is a solid torus. We say $\alpha$ is \textit{Seifert} in $V$ if $V[\alpha]$ is a Seifert-fibered space and not a solid torus. Note that, since $V$ is a genus two handlebody, that $\alpha$ is Seifert in $V$ implies that $V[\alpha]$ is an orientable Seifert-fibered space over $D^2$ with two exceptional fibers, or an orientable Seifert-fibered space over the M\"{o}bius band with at most one exceptional fiber.

Suppose $\alpha$ is a simple closed curve lying in a genus two Heegaard surface $\Sigma$ of $S^3$ bounding
handlebodies $H$ and $H'$. A simple closed curve $\alpha$ in $\Sigma$ is \textit{primitive/primitive}, \textit{double-primitive}, or P/P, if it is primitive in both $H$ and $H'$. Similarly, $\alpha$ is \textit{primitive/Seifert}, or P/SF, if it is primitive in one of $H$ or $H'$, say $H'$, and Seifert in $H$. If a knot $k$ in $S^3$ is represented by a P/P(P/SF, resp.) simple closed curve, then we say that $k$ is a P/P(P/SF, resp.) knot and $(\alpha, \Sigma)$ is a P/P(P/SF, resp.) position of $k$. Note that a knot may have more than one P/P(P/SF, resp.) position. Also note that since $\alpha$ is primitive in one handlebody, a P/P or P/SF knot $k$ is a tunnel-number-one knot in $S^3$.

Let $N(k)$ be a tubular neighborhood of $k$ in $S^3$ and $\gamma$ a component of $\partial N(\alpha) \cap \Sigma$. Then $\gamma$ is an essential simple closed curve in $\partial N(\alpha)$. The isotopy class of $\gamma$ in $\partial N(\alpha)$ defines the \emph{surface slope} in $\partial N(k)$. Note that surface slopes are integral. Then by using, for example, Lemma 2.3 in \cite{D03}, we see that for surface slopes P/P knots admit lens space Dehn surgeries and P/SF knots admit Seifert-fibered Dehn surgeries over $S^2$ with at most three exceptional fibers or over $\mathbb{R}\mathrm{P}^2$ with at most two exceptional fibers. Connected sum of lens spaces might be admitted as Dehn surgeries on P/SF knot, but by \cite{EM92} they occur only from Dehn surgeries on nonhyperbolic P/SF knots.

P/P knots are introduced and studied by Berge \cite{B90}. Berge constructed twelve families of P/P knots, which are referred to as the Berge knots. Then it is proved in \cite{B08} or independently in \cite{G13} that all P/P knots are the Berge knots. This gives the complete classification of P/P knots in $S^3$. However, the hyperbolicity of the Berge knots was undetermined in \cite{B90}.
Later, in \cite{D03}, Dean introduced P/SF knots in $S^3$, which is a natural generalization of P/P knots. The classification of hyperbolic P/SF knots has been performed for years and the complete list of hyperbolic P/SF knots has been made in \cite{BK20}.

In this paper, we provide necessary, sufficient, or equivalent conditions for P/P or P/SF knots being nonhyperbolic, and their applications on some P/SF knots. Therefore the results of this paper will be used in proving the hyperbolicity of P/P and P/SF knots in \cite{K20a} and \cite{K20b} completing the classification of hyperbolic P/P and P/SF knots in \cite{B90} and \cite{BK20} respectively.

Determining whether or not P/P or P/SF knots in $S^3$ are hyperbolic is based on the results of Bleiler-Litherland \cite{BL89} or
independently Wu \cite{W90} for P/P knots, and Miyazaki and Motegi \cite{MM05} for P/SF knots. For P/P knots in $S^3$ which admit lens space surgeries, Bleiler-Litherland \cite{BL89} and independently Wu \cite{W90} showed that if nontrivial surgery on a satellite knot $k$ in $S^3$ yields a lens space, then $k$ is a tunnel-number-one cable of a torus knot, and the surgery is integral. In particular, $k$ is a $(2p_0q_0\pm1, 2)$ cable of a $(p_0, q_0)$ torus knot.

For P/SF knots in $S^3$, Miyazaki and Motegi have shown in \cite{MM05} that if a primitive/Seifert knot $k$ is not hyperbolic, then $k$ is a torus knot or a tunnel-number-one cable of a torus knot. In particular the latter is
an $(sp_0q_0\pm1, s)$-cable of a $(p_0, q_0)$ torus knot

Surgeries on torus knots in $S^3$ yielding lens space or Seifert-fibered space surgeries were completely classified by Moser in \cite{M71}.
Therefore, we have the following classification of nonhyperbolic P/P and P/SF knots.

\begin{thm}\label{nonhyperbolic knots}
If \emph{P/P} or \emph{P/SF} knots are not hyperbolic, then they are torus knots or $(sp_0q_0\pm1, s)$-cable of a $(p_0, q_0)$ torus knot.
\end{thm}

Theorem~\ref{nonhyperbolic knots} enables us to determine nonhyperbolicity of P/P or P/SF knots by investigating a simple closed curve $R$ on the boundary of
a genus two handlebody $H$ such that $H[R]$ is the exterior of a torus knot or a cable of torus knot. This is because if a knot $k$ has a P/P or P/SF position $(\alpha, \Sigma)$
such that $\Sigma$ bounds two handlebodies $H$ and $H'$ and $\alpha$ is assumed to be primitive in $H'$, then there exists a complete set of cutting disks $\{D_M, D_R\}$ of $H'$ such that the boundary $M$ of $D_M$ intersects $\alpha$ once and the boundary $R$ of $D_R$ is disjoint from $\alpha$. Note that such a cutting disk $D_R$ disjoint from $\alpha$ is unique in $H'$ up to isotopy. Since $M$ intersects $\alpha$ once and $R$ is disjoint from $\alpha$, $M$ is a meridian curve of $k$ and $H[R]$ is homeomorphic to
the exterior of $k$ in $S^3$. Therefore, the hyperbolicity of $k$ is determined by the hyperbolicity of $H[R]$.

In order to obtain necessary, sufficient, or equivalent conditions for P/P or P/SF knots being nonhyperbolic in this paper, we investigate a simple closed curve $R$ on the boundary of a genus two handlebody $H$ such that $H[R]$ is the exterior of a torus knot or a cable of torus knot by using some results of \cite{K20c} and \cite{K20d}.

Throughout the paper, we need the background for the three diagrams: Heegaard diagrams, R-R diagrams, and hybrid diagrams. See Section 2 in \cite{K20d} for their backgrounds.\\

\noindent\textbf{Acknowledgement.} In 2008, in a week-long series of talks to a seminar in the department of mathematics of the University of Texas as Austin, John Berge outlined a project to completely classify and describe the primitive/Seifert knots in $S^3$. The present paper, which provides some of the background materials necessary to carry out the project, is originated from the joint work with John Berge for the project. I should like to express my gratitude to John Berge for his support and collaboration. I would also like to thank Cameron Gordon and John Luecke for their support while I stayed in the University of Texas at Austin.

\section{Torus or cable knot relators}\label{Nonhyperbolic of P/P and P/SF knots}

Suppose $\alpha$ is P/P or P/SF in a genus two Heegaard splitting $(\Sigma; H, H')$ of $S^3$
assuming that $\alpha$ is primitive in $H'$ in both cases. Let $k$ be a knot in $S^3$ represented by $\alpha$. As explained in the introduction, since $\alpha$ is primitive in $H'$, there is a unique cutting disk $D_R$ in $H'$ such that the boundary $R$ of $D_R$ is disjoint from $\alpha$ and $H[R]$ is homeomorphic to
the exterior of $k$ in $S^3$. This implies that $k$ is a tunnel-number-one knot
such that $D_R$ is a co-core of a 1-handle regular neighborhood of some unknotting tunnel of $k$. Equivalently, the pair $(H, R)$ provides a genus two Heegaard splitting of the exterior of $k$.

Unknotting tunnels(or equivalently genus two Heegaard splittings of the exterior) of torus knots and cable of torus knots are completely
understood by Boileau-Rost-Zieschang \cite{BRZ88} or independently by Moriah \cite{M88}
for torus knots and by \cite{MS91} for cable of torus knots. This enables us to have
the following, which are results of \cite{K20d}.

\begin{thm}\label{torus knot relator R}
Suppose $R$ is a simple closed curve in the boundary of a genus two handlebody $H$ such that $H[R]$ embeds in $S^3$.
Then $H[R]$ is the exterior of a torus knot in $S^3$ if and only if $R$ has an R-R diagram of the form shown in Figure~\emph{\ref{PSFFig3aa}} with $n, s>1$, $a,b>0$, $\gcd(a,b)=1$.

If $R$ has an R-R diagram of the form shown in Figure~\emph{\ref{PSFFig3aa}a}\emph{(\ref{PSFFig3aa}b, \textit{resp}.)}, then $H[R]$ is the exterior of an $(n, s)$ with $n>s$$(((a+b)n+b, s), \text{resp}.)$ torus knot.
\end{thm}
\begin{proof}
This follows from Lemma 5.16, Proposition 5.18, and Proposition 5.21 in \cite{K20d}.
\end{proof}

\begin{thm}\label{cable knot relator R}
Suppose $R$ is a simple closed curve in the boundary of a genus two handlebody $H$ such that $H[R]$ embeds in $S^3$.
Then $H[R]$ is the exterior of a cable knot if and only if
$R$ has an R-R diagram of the form shown in Figure~\emph{\ref{PSFFig3ab}}, with $m, n, s > 1$, $a, b > 0$, $\gcd(a,b)$ = $\gcd(m,n) = 1$, and $n(a+b) + bm$ = $sm(a+b) + \pm1$.

If $R$ has an R-R diagram of the form shown in Figure~\emph{\ref{PSFFig3ab}}, then $H[R]$ is the exterior of
a $(sp_0q_0\pm1,s)$ cable of a $(p_0,q_0)$ torus knot in $S^3$, where $\{p_0, q_0\}=\{m, a+b\}$.
\end{thm}

\begin{proof}
This follows from Lemma 5.16, Proposition 5.18, and Proposition 5.22 in \cite{K20d}.
\end{proof}

\begin{figure}[tbp]
\centering
\includegraphics[width = 1\textwidth]{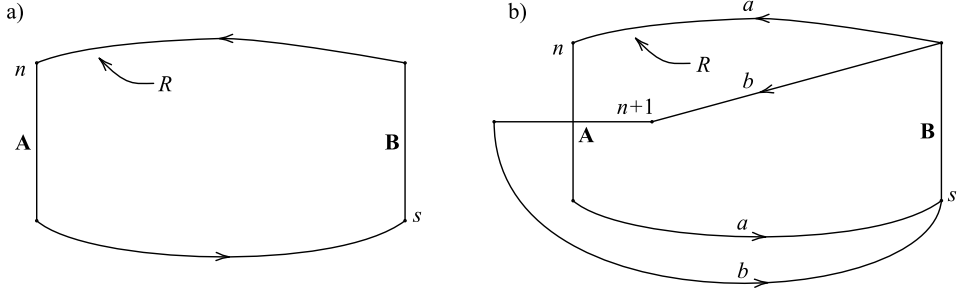}
\caption{Suppose $R$ is a simple closed curve in the boundary of a genus two handlebody $H$ such that $R$ has an R-R diagram
as in this figure, with $n, s>1$, $a,b>0$, $\gcd(a,b)=1$. Then $H[R]$ is the exterior of an $(n, s)$ torus knot with $n>s$(Figure~\ref{PSFFig3aa}a) or
an $((a+b)n+b, s)$ torus knot (Figure~\ref{PSFFig3aa}b). The converse is also true.}
\label{PSFFig3aa}
\end{figure}

\begin{figure}[tbp]
\centering
\includegraphics[width = 0.65\textwidth]{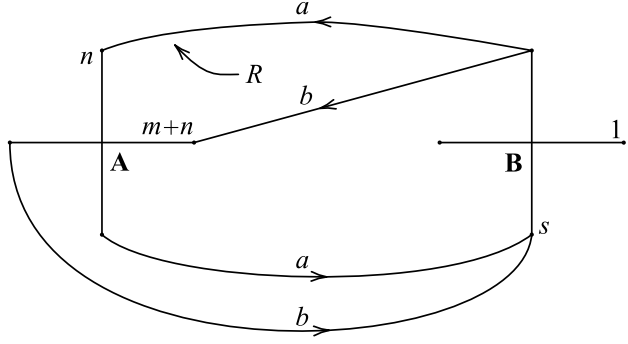}
\caption{Suppose $R$ is simple closed curve in the boundary of a genus two handlebody $H$ such that $R$ has an R-R diagram as in this figure, with $m, n, s > 1$, $a, b > 0$, $\gcd(a,b)$ = $\gcd(m,n) = 1$, and $n(a+b) + bm$ = $sm(a+b) + \delta$ with $\delta = \pm 1$. Then $H[R]$ is the exterior of aa $(sp_0q_0\pm1,s)$ cable of a $(p_0,q_0)$ torus knot in $S^3$, with $\{p_0, q_0\}=\{m, a+b\}$. The converse is also true.}
\label{PSFFig3ab}
\end{figure}

If a simple closed curve $R$ on a genus two handlebody $H$ has the diagram of the form shown in Figure~\ref{PSFFig3aa}, then $R$ is said to be \textit{a torus knot relator}. In particular, if $R$ has the diagram in Figure~\ref{PSFFig3aa}a (resp. Figure~\ref{PSFFig3aa}b), then $R$ is called \textit{a rectangular(resp. non-rectangular) torus knot relator}.
If a simple closed curve $R$ has the diagram of the form shown in Figure~\ref{PSFFig3ab}, then $R$ is called \textit{a cable knot relator}.

\begin{prop}
Suppose $R$ is a simple closed curve in the boundary of a genus two handlebody $H$ such that $H[R]$ embeds in $S^3$.
Let $\{D_A, D_B\}$ be a complete set of cutting disks of $H$ in which the Heegaard diagram $\mathbb{D}$ of $R$ is nonpositive.
If $H[R]$ is the exterior of a torus or a cable of a torus knot in $S^3$, then an edge $e$ can be added in $\mathbb{D}$
such that its endpoints lie at one cutting disk, and $R$ and $e$ can be oriented so that all of their intersections have the
same sign.
\end{prop}

\begin{proof}
Theorems~\ref{torus knot relator R} and \ref{cable knot relator R} imply that $R$ has
an R-R diagram of the form in Figure~\ref{PSFFig3aa} or \ref{PSFFig3ab}. Let $\{D'_A, D'_B\}$ be a complete set of cutting disks of $H$ in which the Heegaard diagram $\mathbb{D}'$ of $R$ underlies the R-R diagram of $R$ in Figure~\ref{PSFFig3aa} or \ref{PSFFig3ab}. Then $\mathbb{D}'$ is positive, i.e., all of the intersections in $R\cap \partial D'_A$ and $R\cap \partial D'_B$ have the same sign.

Since the Heegaard diagram $\mathbb{D}$ of $R$ with respect to $\{D_A, D_B\}$ is nonpositive, $(D'_A\cup D'_B)\cap (D_A\cup D_B)$
contains essential intersections. We may assume that $D'_A\cap(D_A\cup D_B)$ has essential intersections. Then
there must exist an outermost subdisk $D$ of $D'_A$ cut off by essential intersections of $D'_A\cap(D_A\cup D_B)$.
Since all of the intersections in $R\cap \partial D'_A$ have the same sign,
the arc $D\cap\partial H$ has only one type of signed intersection with $R$. This completes the proof.
\end{proof}

The parameters $m, n,$ and $s$ in the diagram of a torus or a cable knot relator $R$ in Figure~\ref{PSFFig3aa} or \ref{PSFFig3ab} are restricted to positive integers. The following theorem shows that if the parameters are extended to any integers, then $H[R]$ is still nonhyperbolic.

\begin{thm}\label{at most two bands and one band}
Suppose $R$ is a simple closed curve in the boundary of a genus two handlebody $H$ with an R-R diagram of the form shown in Figure~\emph{\ref{PSFFig3ae}} with $a, b\geq 0$
and $m, n, s\in \mathbb{Z}$.

If $H[R]$ embeds in $S^3$, then $R$ is either a primitive curve, or a torus or a cable knot relator in $H$. Therefore if $k$ is
a knot whose exterior is homeomorphic to $H[R]$, then $k$ is either the unknot, a torus knot or a cable of a torus knot.
\end{thm}

\begin{figure}[tbp]
\centering
\includegraphics[width = 0.6\textwidth]{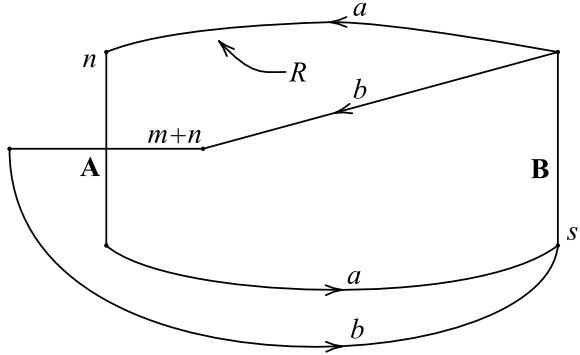}
\caption{An R-R diagram of $R$ where $R$ has only one band of connections in one handle and has at most two bands of connections in the other handle and $a, b\geq 0$
and $m, n, s\in \mathbb{Z}$.}
\label{PSFFig3ae}
\end{figure}

\begin{proof}
Consider the R-R diagram of $R$ in Figure~\ref{PSFFig3ae}. Then we may assume without loss of generality that $m, s\geq0$.

If $m=0$, then $n=1$. This implies that $R=(AB^s)^{a+b}$ in $\pi_1(H)$. Since $H[R]$ embeds in $S^3$, for the homological reason, $a+b=1$ and thus $R$ is primitive. If $s=0$, then $R=A^{(a+b)n+bm}$ in $\pi_1(H)$. For the homological reason, $|(a+b)n+bm|=1$ and $R$ is primitive. If $ab=0$, then without loss of generality, we may assume that $a=1$ and $b=0$ and thus the diagram of $R$ is a rectangular torus knot relator
unless $|n|=1$ or $s=1$, in which case $R$ is primitive.

Thus we may assume that $m, s > 0$ and $ab>0$. Let $a=\rho b +r$, where $\rho\geq 0$ and $0\leq r<b$ (if $r=0$, then $\rho>0$ and $b=1$). We divide the argument into two subcases: (1) $s=1$ and (2) $s>1$.

\begin{figure}[tbp]
\centering
\includegraphics[width = 1\textwidth]{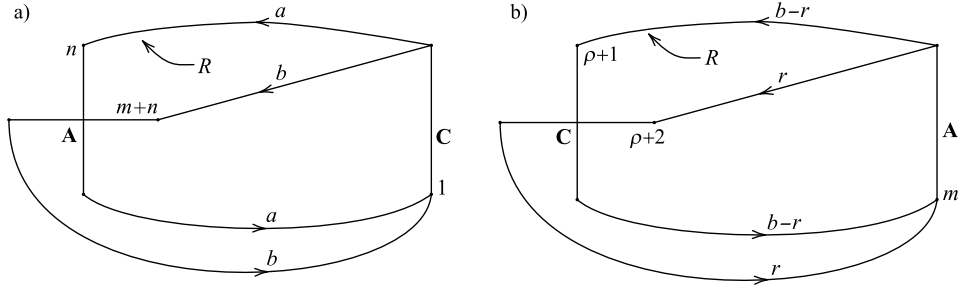}
\caption{An R-R diagram of $R$ with $s=1$ and new R-R diagram of $R$ after the change of cutting disks.}
\label{PSFFig3ac}
\end{figure}

(1) Suppose $s=1$. Then the R-R diagram of $R$ is given in Figure~\ref{PSFFig3ac}a. It follows that $R$ is the product of two subwords $BA^{n}$ and $BA^{m+n}$ with $|BA^{n}|=a$ and $|BA^{m+n}|=b$. Here, $|BA^{n}|$ and $|BA^{m+n}|$ denote the total number of appearances of $BA^{n}$ and $BA^{m+n}$ respectively in the word of $R$ in $\pi_1(H)=F(A,B)$. We perform a change of cutting disks of $H$ underlying the diagram, which induces an automorphism $\theta$ of $\pi_1(H)$ that takes $B \mapsto BA^{-n}$ and leaves $A$ fixed. Then the new R-R diagram realized by the resulting Heegaard diagram of $R$ is shown as in Figure~\ref{PSFFig3ac}b. If $m>1$ in the diagram, then $R$ belongs to a torus knot relator. If $m=1$, then  when $r=0$, $R=BA^{\rho+1}$ in $\pi_1(H)$ and thus is primitive, and when $r>0$, by Lemma 3.5 of \cite{K20c}, $R$ is also primitive in $H$.

(2) Suppose $s>1$. If $n(m+n)>0$, then by Theorems~\ref{torus knot relator R} and \ref{cable knot relator R} $R$ is a torus or a cable knot relator. If $n(m+n)<0$, then the underlying Heegaard diagram of $R$ is nonpositive and thus there exists a vertical
wave based at $R$. By one of the results of \cite{B20}, which is originated from \cite{B93}, one meridian curve of $H[R]$ is obtained by the surgery on $R$ along the vertical wave. We can see from the R-R diagram of $R$ that such meridian curve of $H[R]$ is the curve whose first homotopy is $B^s$ in $\pi_1(H)$. This is impossible since $s>1$ and $H[M]$ embeds in $S^3$. Therefore $n(m+n)\geq0$.

\begin{figure}[tbp]
\centering
\includegraphics[width = 1\textwidth]{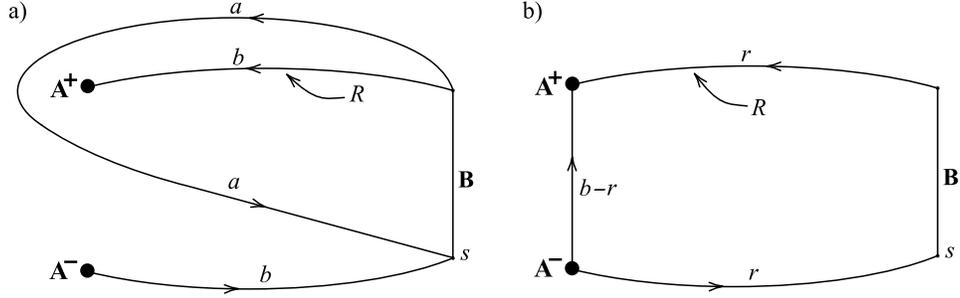}
\caption{The Hybrid diagram corresponding to the R-R diagram in Figure~\ref{PSFFig3ae} with $(m, n)=(1,0)$ and
the hybrid diagram after performing a change of cutting disks inducing an automorphism of $\pi_1(H)$ that takes $A \mapsto AB^{-(\rho+1)s}$. }
\label{PSFFig3ad}
\end{figure}

Now suppose $n(m+n)=0$. Note that both $n$ and $m+n$ cannot be $0$. Therefore without loss of generality we may assume that $n+m=1$ and then $n=0$.
Since $n=0$, there is a $0$-connection in the diagram of $R$. If $r=0$ in the equation $a=\rho b+r$, then $b=1$ and $R=AB^{(\rho+1)s}$, which is primitive.
So we may assume that $r>0$ in the equation $a=\rho b+r$.

Now we use the argument of the hybrid diagram. Its hybrid diagram of $R$ corresponding to the R-R diagram in Figure~\ref{PSFFig3ae} with $(m, n)=(1,0)$ is illustrated in Figure~\ref{PSFFig3ad}a. In its hybrid diagram, we drag $\rho+1$ times the vertex $A^-$
together with the edges meeting with the vertex $A^-$ over the $s$-connection on the $B$-handle. This performance
corresponds to a change of cutting disks inducing an automorphism of $\pi_1(H)$ that takes $A \mapsto AB^{-(\rho+1)s}$.
The resulting hybrid diagram of $R$ is depicted in Figure~\ref{PSFFig3ad}b, where there is only one band of connections labelled by $s$ in the $B$-handle.
For the $A$-handle in the corresponding R-R diagram of $R$, since $R$ intersects the cutting disk $D_A$ positively and $b-r>0$, $R$ is positive and thus there are at most three bands of connections and one of the bands of connections has a label greater than $1$. If there are three bands of connections in $A$-handle, then by Theorem 5.1 in \cite{K20d}, $H[R]$ cannot embed in $S^3$. If there are at most two bands of connections in $A$-handle, then $R$ is
a torus or a cable knot relator.
\end{proof}

Theorem~\ref{at most two bands and one band} says that if $R$ has only one connection on one handle and has at most two connections on the other handle, then
$H[R]$ is not hyperbolic.

\section{Presentations of a torus or a cable knot relator}\label{Presentations of a torus or a cable a torus knot relator}

Suppose $H$ is a genus two handlebody with a complete set of cutting disks $\{D_A, D_B\}$ with $\pi_1(H)=F(A,B)$, where
the generators $A$ and $B$ are dual to $D_A$ and $D_B$ respectively. Suppose $R$ is a simple closed curve on $\partial H$ such that $H[R]$ embeds in $S^3$.
Then $R$ represents a relator in $\pi_1(H)=F(A,B)$, i.e., $\pi_1(H[R])=\langle A, B \,|\, R\rangle$.
In this section, we deal with some backgrounds of two generator presentations and equivalent presentation conditions for $R$ being a torus or a cable knot relator.

\subsection{Backgrounds of two generator presentations} \hfill

\smallskip

Suppose $\mathcal{P}$ is a finite presentation of a two generator set $G$. Let $F$ be the free group freely generated by the two generators in $G$.
If $\mathcal{R}$ is a relator of a presentation $\mathcal{P}$ which is freely and cyclically reduced, then the \emph{algebraic length} of $\mathcal{R}$ is the total number of letters which appear in $\mathcal{R}$.
The \emph{algebraic length} of $\mathcal{P}$, denoted by $|\mathcal{P}|$, is the sum of the lengths of all of the relators in $\mathcal{P}$. The presentation $\mathcal{P}$ has \emph{minimal algebraic length} under automorphisms of $F$ if applying any automorphism of $F$ to $\mathcal{P}$ and then freely and cyclically reducing the relators of the result, yields a presentation $\mathcal{P}'$ with $|\mathcal{P}'| \geq |\mathcal{P}|$.

For a free group $F(A, B)$ generated by $A$ and $B$, the automorphisms
\begin{equation}\label{Four Whitehead automorphisms1}
\begin{aligned}
(A,B) \mapsto (AB^{\pm1},B)\\
(A,B) \mapsto (A,BA^{\pm1})
\end{aligned}
\end{equation}

\noindent of $F(A, B)$ are called \textit{T-transformations} or \textit{Whitehead automorphisms}.

Suppose $\mathcal{P}$ is a presentation which has minimal algebraic length. If a Whitehead automorphism $\theta$ transforms $\mathcal{P}$ into
a presentation $\mathcal{P}'$ such that $|\mathcal{P}'| = |\mathcal{P}|$, then $\theta$ is called a \textit{level-T-transformation} or simply a \textit{level-transformation}.

The following theorems are well-known results of Whitehead \cite{W36}.

\begin{thm}
\label{min length under auts}
Let $\mathcal{P}$ and $F$ be defined as above. Suppose that all the relators in $\mathcal{P}$ are freely
and cyclically reduced. If there is an automorphism of the free group $F$ which reduces the length of $\mathcal{P}$, then there exists
a Whitehead automorphism which will also reduce the length of $\mathcal{P}$.
\end{thm}

\begin{thm}
\label{finite sequence of level-T-transformation}
Let $\mathcal{P}$, $G$, and $F$ be defined as above and $\mathcal{P}'$ be another finite presentation of $G$
such that the relators of both $\mathcal{P}$ and $\mathcal{P}'$ are freely
and cyclically reduced. Suppose both $\mathcal{P}$ and $\mathcal{P}'$ have minimal
length, and suppose there exists an automorphism $\pi$ of $F$ such that $\pi(\mathcal{P})=\mathcal{P}'$. Then there exists
a finite sequence of level-transformations which also transforms $\mathcal{P}$ to $\mathcal{P}'$.
\end{thm}

Suppose $\mathcal{P} = \langle \, A,B \,|\, \mathcal{R}_1, \dots ,\mathcal{R}_m \,\rangle$ is a two generator finitely related presentation. The \emph{Whitehead graph} $G_W(\mathcal{P})$ of $\mathcal{P}$ is defined as follows. First, $G_W(\mathcal{P})$ has four vertices, labeled $A^+$, $A^-$, $B^+$, and $B^-$. Then $G_W(\mathcal{P})$ has one unoriented edge joining vertex $x$ and vertex $y^{-1}$ of $G_W(\mathcal{P})$ corresponding to each appearance of $xy$ as a subword of a relator of $\mathcal{P}$. Note that if $\mathcal{R}$ is a relator of $\mathcal{P}$ and $\mathcal{R}$ = $x_1x_2 \dots x_n$, with $x_i \in \{A, A^-, B, B^-\}$, then the pair $x_n x_1$ also contributes an edge to $G_W(\mathcal{P})$. Note also that if $\mathcal{R} = X$, with $X$ = $A^{\pm 1}$ or $B^{\pm 1}$, then $\mathcal{R}$ contributes one edge, joining vertex $X^+$ and vertex $X^-$, to $G_W(\mathcal{P})$.

A \emph{loop} in $G_W(\mathcal{P})$ is an edge of $G_W(\mathcal{P})$ with both of its endpoints at the same vertex of $G_W(\mathcal{P})$. Thus, if the relators of $\mathcal{P}$ are freely and cyclically reduced, then each edge of $G_W(\mathcal{P})$ has its endpoints at distinct vertices of $G_W(\mathcal{P})$, and there are no loops in $G_W(\mathcal{P})$.

A finite graph $G$ is \emph{connected} if given any two distinct vertices, say $V$ and $V'$, of $G$, there is a path consisting of a finite union of edges of $G$ connecting $V$ to $V'$.

A vertex $V$ of a finite connected graph $G$ is a \emph{cut-vertex} of $G$ if deleting $V$, and the edges of $G$ which meet $V$, from $G$ yields a graph $G'$, which is not connected.

\begin{figure}[tbp]
\centering
\includegraphics[width = 1.0\textwidth]{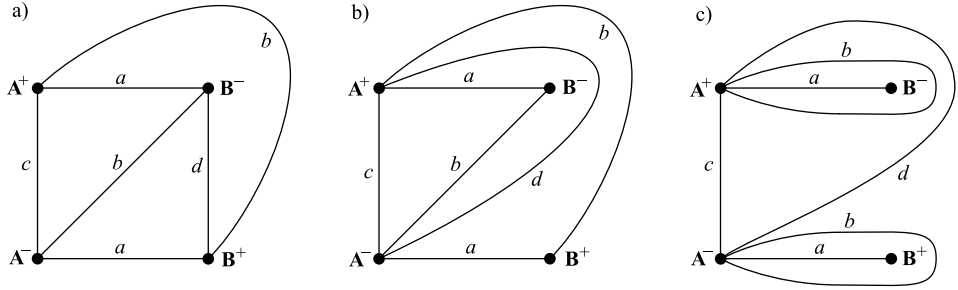}
\caption{The three types of graphs of Heegaard diagrams of simple closed curves on
the boundary of a genus two handlebody $H$ which have essential intersections with cutting disks $D_A$ and $D_B$.}
\label{DPCFig8a-3}
\end{figure}

\begin{prop}\label{min length conditions}
Suppose $\mathcal{P}$ is a two-generator presentation, which has a Whitehead graph of the form of Figure~\emph{\ref{DPCFig8a-3}a}. Then
\begin{enumerate}
\item $\mathcal{P}$ has minimal algebraic length under automorphisms if and only if $|a-b| \leq c$ and $|a-b| \leq d$.
\item $\mathcal{P}$ has a level-transformation if and only if either $|a-b|=c$ or $|a-b|=d$.
\end{enumerate}
Furthermore, if $|a-b|<c$ and $|a-b|=d(|a-b|=c$ and $|a-b|<d$, resp.$)$, then only level-transformations are the Whitehead automorphisms $(A,B) \mapsto (A,BA^{\pm1})$
$((A,B) \mapsto (AB^{\pm1},B)$, resp.$)$.
\end{prop}

\begin{proof}
Theorems~\ref{min length under auts} implies that we need only to see how algebraic length of $\mathcal{P}$ changes under the Whitehead automorphisms.
The algebraic length of $\mathcal{P}$ in Figure~\ref{DPCFig8a-3}a is $2a+2b+c+d$. It is not hard to show that if we apply the Whitehead automorphism $\theta_1: (A,B) \mapsto (AB,B)$ to $\mathcal{P}$, then algebraic length of $\theta_1(\mathcal{P})$ is $3a+b+2c+d$. Similarly, the Whitehead automorphisms $\theta_2: (A,B) \mapsto (AB^{-1},B)$,
$\theta_3: (A,B) \mapsto (A,BA)$, and $\theta_4: (A,B) \mapsto (A,BA^{-1})$ change the algebraic length of $\mathcal{P}$ to $a+3b+2c+d$, $3a+b+c+2d$, and $a+3b+c+2d$ respectively.
Therefore $2a+2b+c+d$ is minimal algebraic length of $\mathcal{P}$ if and only if

\begin{equation}\label{Four inequalities}
\begin{aligned}
2a+2b+c+d \leq 3a+b+2c+d(\Leftrightarrow b-a\leq c),\\
2a+2b+c+d \leq a+3b+2c+d(\Leftrightarrow a-b\leq c),\\
2a+2b+c+d \leq 3a+b+c+2d(\Leftrightarrow b-a\leq d),\\
2a+2b+c+d \leq a+3b+c+2d(\Leftrightarrow a-b\leq d).
\end{aligned}
\end{equation}

\noindent Then the inequalities in (\ref{Four inequalities}) complete the proof.
\end{proof}

Let $H$ be a genus two handlebody with a complete set of cutting disks $\{D_A, D_B\}$ with $\pi_1(H)=F(A,B)$,
$\mathcal{R}=\{R_1, \ldots, R_n\}$ be a set of pairwise disjoint simple closed curves on $\partial H$ which have essential intersections with $D_A\cup D_B$.
Let $\mathbb{D}_\mathcal{R}$ be a Heegaard diagram of $\mathcal{R}$ with respect to $\{D_A, D_B\}$ and $[R_i]$ be a presentation represented by $R_i$ in $\pi_1(H)=F(A,B)$. Then we have a two-generator presentation $\mathcal{P}_\mathcal{R} = \langle \, A,B \,|\, [R_1], \ldots,  [R_n] \,\rangle$. The \emph{geometric length} of the Heegaard diagram $\mathbb{D}_\mathcal{R}$, denoted by $|\mathbb{D}_\mathcal{R}|$ is defined to be the total number of edges in $\mathbb{D}_\mathcal{R}$.

The following theorems are well-known results of Zieschang \cite{Z70}.

\begin{thm}
\label{geo min length under auts}
Let $H, \mathcal{R}$, and $\mathcal{P}_\mathcal{R}$ be defined as above. If there is an automorphism of the free group $F(A,B)$ which reduces the algebraic length of $\mathcal{P}_\mathcal{R}$,
then there exists a change of cutting disks of $H$ inducing one of the Whitehead automorphisms in \emph{(\ref{Four Whitehead automorphisms1})}
which also reduce the geometric length of $\mathbb{D}_\mathcal{R}$.
\end{thm}

Note that a change of cutting disks inducing a Whitehead automorphism can be achieved by replacing
one cutting disk $D$ by a cutting disk $D'$ which is obtained by bandsumming $D$ with another cutting disk along an arc connecting them.
In terms of the terminology in \cite{W36} and \cite{Z70}, such changes of cutting disks of $H$ described in Theorem~\ref{geo min length under auts} is called \textit{geometric-T-transformation} of $H$. If a geometric-T-transformation of $H$ carries a geometric minimal Heegaard diagram $\mathbb{D}$ to another Heegaard digarm $\mathbb{D}'$ with $|\mathbb{D}'|=|\mathbb{D}|$,
then it is called a \textit{level-geometric-T-transformation}, or simply \textit{level-geometric-transformation}.

\begin{thm}
\label{geometric minimal realization}
Let $H, \mathcal{R}$, $\mathbb{D}_\mathcal{R}$, and $\mathcal{P}_\mathcal{R}$ be defined as above. Suppose $\mathcal{P}_\mathcal{R}$ has minimal algebraic length, $|\mathbb{D}_\mathcal{R}|=|\mathcal{P}_\mathcal{R}|$, and $\mathcal{P'}_\mathcal{R}$ is a presentation obtained from $\mathcal{P}_\mathcal{R}$ by a level-transformation.
Then there exists a level-geometric-transformation which carries $\mathbb{D}_\mathcal{R}$ into $\mathbb{D'}_\mathcal{R}$ such that $|\mathbb{D'}_\mathcal{R}|=|\mathcal{P'}_\mathcal{R}|$,
and $\mathbb{D'}_\mathcal{R}$ realizes $\mathcal{P'}_\mathcal{R}$.
\end{thm}

\begin{cor} \label{unique under no level trans}
Let $H, \mathcal{R}$, $\mathbb{D}_\mathcal{R}$, and $\mathcal{P}_\mathcal{R}$ be defined as above.
Suppose that $\mathcal{P}_\mathcal{R}$ has minimal algebraic length.
If $\mathcal{P}_\mathcal{R}$ has no level-transformations, then the complete set of cutting disks $\{D_A, D_B\}$ is unique such that $\mathcal{P}_\mathcal{R}$ is minimal.
\end{cor}
\begin{proof}
This follows from Theorem~\ref{geometric minimal realization}.
\end{proof}

Regarding Proposition~\ref{min length conditions}, we note that if $\mathbb{D}_\mathcal{R}$ has a graph of the form of Figures~\ref{DPCFig8a-3}a or \ref{DPCFig8a-3}b, then $\mathbb{D}_\mathcal{R}$ realizes $\mathcal{P}_\mathcal{R}$.
Therefore the graph $\mathbb{D}_\mathcal{R}$ of the form of Figure~\ref{DPCFig8a-3}a or and \ref{DPCFig8a-3}b is the Whitehead graph of $\mathcal{P}_\mathcal{R}$.

\subsection{Presentations of a torus or a cable knot relator} \hfill

\smallskip

Suppose $R$ is a simple closed curve in the boundary of a genus two handlebody $H$ such
that $H[R]$ embeds in $S^3$. We have shown that if $H[R]$ is the exterior of a torus or a cable knot in $S^3$,
then $R$ has an R-R diagram with the form of Figure~\ref{PSFFig3aa} or \ref{PSFFig3ab}.
The following theorem provides a strategy to
determine whether or not $H[R]$ is the exterior of a torus or a cable knot.

\begin{figure}[t]
\centering
\includegraphics[width = 1\textwidth]{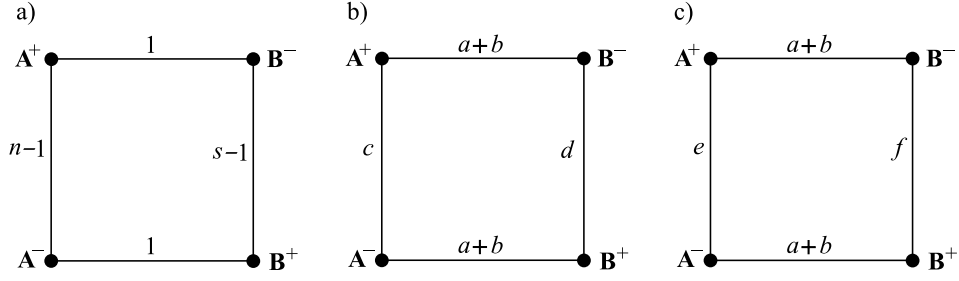}
\caption{The Whitehead graphs of $R$ corresponding to Figures~\ref{PSFFig3aa} and \ref{PSFFig3ab}.}
\label{WHgraph-3}
\end{figure}

\begin{thm}
Suppose $R$ is a simple closed curve in the boundary of a genus two handlebody $H$ such
that $H[R]$ is the exterior of a torus or a cable knot in $S^3$,
and a complete set of cutting disks $\{D_A, D_B\}$ of $H$ intersects $R$ minimally.
Let $\mathcal{P}_R$ be the two-generator presentation $\pi_1(H[R])=\langle \, A,B \,|\, [R]\,\rangle$

If $R$ has no level-transformations in $\mathcal{P}_R$, then the R-R diagram
of $R$ with respect to $\{D_A, D_B\}$ has the form of Figure~\emph{\ref{PSFFig3aa}} or \emph{\ref{PSFFig3ab}}
up to the homeomorphisms of $H$ inducing the automorphisms
exchanging $A$ and $B$, and replacing $A^{-1}$ by $A$ or $B^{-1}$ by $B$.

If $R$ has a level-transformation in $\mathcal{P}_R$, then there is an automorphism on $\pi_1(H)=F(A,B)$
of form $A\mapsto AB^p$ or $B\mapsto BA^q$ which carries $R$ to a curve with an R-R diagram of the form
of Figure~\emph{\ref{PSFFig3aa}} or \emph{\ref{PSFFig3ab}} with $s=2$ up to the homeomorphisms of $H$ inducing the automorphisms
exchanging $A$ and $B$, and replacing $A^{-1}$ by $A$ or $B^{-1}$ by $B$.
\end{thm}

\begin{proof}
The Whitehead graphs of $R$ corresponding to Figures~\ref{PSFFig3aa} and \ref{PSFFig3ab} appear as Figure~\ref{WHgraph-3} respectively,
where $c=(n-1)(a+b)+b, d=(s-1)(a+b)$ in Figure~\ref{WHgraph-3}b and $e=(n-1)(a+b)+mb, f=(s-1)(a+b)$ in Figure~\ref{WHgraph-3}c.
The conditions on $n, m, s$ given by Theorems~\ref{torus knot relator R} and \ref{cable knot relator R}, and Proposition~\ref{min length conditions} imply that
$R$ has minimal length under automorphisms and has no level-transformations unless $s=2$. It follows from Corollary~\ref{unique under no level trans} that the set of cutting disks $\{D_A, D_B\}$ of $H$ yielding the R-R diagrams of Figures~\ref{PSFFig3aa} and \ref{PSFFig3ab} is unique such that $R$ has minimal. Therefore $R$ in the hypothesis must have the R-R diagram of Figure~\ref{PSFFig3aa} or \ref{PSFFig3ab}.

Now we assume that $s=2$ in each R-R diagram of Figures~\ref{PSFFig3aa} and \ref{PSFFig3ab}. Then $R$ has minimal length under automorphisms and
has level-transformations.

First consider Figure~\ref{WHgraph-3}a with $s=2$. Then $R=A^nB^2$. It is easy to see by applying the Whitehead automorphisms that the Whitehead automorphism $\theta: B\mapsto BA^{-1}$ is only level-transformation. Therefore the automorphisms keeping the algebraic length of $R$ unchanged are of the form $\theta^p=B\mapsto BA^{-p}$ with $1\leq p \leq n$. If $p=n$, then the automorphism $B\mapsto BA^{-p}$ carries $R=A^nB^2$ to $A^{-n}B^2$, which is the same form. Therefore we may assume $1\leq p < n$.
Now we need to show that the only automorphisms of the image of $R$ under $B\mapsto BA^{-p}$ which leaves the algebraic length of $R$ unchanged are of the form $B\mapsto BA^{q}$.
However this is true by Proposition~\ref{min length conditions} because the image of $R$ under $B\mapsto BA^{-p}$ is $A^{n-p}BA^{-p}B$, whose Whitehead graph has the form shown in Figure~\ref{WHgraph2-3}a.

\begin{figure}[t]
\centering
\includegraphics[width = 1\textwidth]{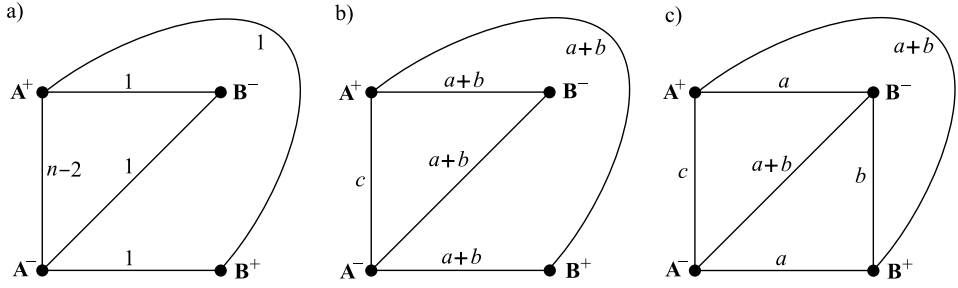}
\caption{The Whitehead graphs of of the image of $R$ under the automorphism $B\mapsto BA^{-p}$.}
\label{WHgraph2-3}
\end{figure}

Second consider Figure~\ref{WHgraph-3}b with $s=2$, where $c=(n-1)(a+b)+b, d=(s-1)(a+b)$. Then $|B^2A^n|=a$ and $|B^2A^{n+1}|=b$ and
it follows by applying the Whitehead automorphisms that the only level-transformations fixing the algebraic length of $R$
are of the form $B\mapsto BA^{-p}$ with $1\leq p \leq n$.
On the other hand, $BA^{-p}BA^{n-p}$ with $|BA^{-p}BA^{n-p}|=a$ and $BA^{-p}BA^{n+1-p}$ with $|BA^{-p}BA^{n+1-p}|=b$ appear in the image of $R$ under the automorphism $B\mapsto BA^{-p}$. Its Whitehead graph appears as in Figure~\ref{WHgraph2-3}b with $c=(n-2)(a+b)+b$ or Figure~\ref{WHgraph2-3}c with $c=(n-1)(a+b)$ depending on the value of $p$, i.e., $p<n$ or $p=n$. In either of the Whitehead graphs, Proposition~\ref{min length conditions}
implies that the only automorphisms of the image of $R$ under $B\mapsto BA^{-p}$ which leaves the algebraic length of $R$ unchanged are of the form $B\mapsto BA^{q}$.

Lastly, we consider Figure~\ref{WHgraph-3}c with $s=2$, where $e=(n-1)(a+b)+mb, f=(s-1)(a+b)$. However if $m=1$, then this becomes the second case. Therefore
we can apply the similar argument to show that the only level-transformations fixing the algebraic length of $R$
are of the form $B\mapsto BA^{-p}$ with $1\leq p \leq n$, and the only automorphisms of the image of $R$ under $B\mapsto BA^{-p}$ which leaves the algebraic length of $R$ unchanged are of the form $B\mapsto BA^{q}$.
\end{proof}

\section{Proper power curves and Nonhyperbolicity}\label{Proper power curves and Nonhyperbolicity}

A simple closed curve $\beta$ in the boundary of a genus two handlebody $H$ is said to be a \textit{proper power curve}
if $\beta$ is disjoint from a separating disk in $H$, does not bound a disk in $H$, and is not
primitive in $H$. Equivalently, $\beta$ represents a proper power of a primitive element in $\pi_1(H)$.
Proper power curves play a very essential role in the classification of hyperbolic P/P and P/SF knots.
In particular, the following theorem says that nonhyperbolicity of P/P and P/SF knots is equivalent to the existence of proper
power curves.

\begin{thm}\label{nonhyperbolicity from a proper power}
Suppose $R$ is a simple closed curve in the boundary of a genus two handlebody $H$ such that $H[R]$ embeds in $S^3$.
Then $H[R]$ is the exterior of the unknot, a torus knot, or a cable knot if and only if there exists a proper power curve $\beta$ in $H$ disjoint from $R$.
\end{thm}
\begin{proof}
This is Theorem 5.26 in \cite{K20d}.
\end{proof}

\begin{prop}\label{nonhyperbolicity of P/P and P/SF from a proper power}
Let $\alpha$ be a $\emph{P/P}$ or $\emph{P/SF}$ simple closed curve in a genus two Heegaard splitting $(\Sigma; H, H')$ of $S^3$ assuming that $\alpha$ is primitive in $H'$ in both cases, and $k$ be a knot in $S^3$ represented by $\alpha$.
Let $R$ be a unique simple closed curve in $\partial H$ such that $R$ is disjoint from $\alpha$ and bounds a cutting disk of $H'$.

Then $k$ is not hyperbolic if and only if there exists a proper power curve $\beta$ in $H$ disjoint from $R$.
\end{prop}
\begin{proof}
Assume that $k$ is not hyperbolic. By Theorem~\ref{nonhyperbolic knots}, $k$ is a torus or cable of a torus knot. Note that in Theorem~\ref{nonhyperbolic knots}, the unknot is considered as a torus knot. Since $H[R]$ is the exterior of $k$ in $S^3$, $H[R]$ is the unknot, a torus knot, or a cable of a torus knot. Therefore $k$ is not hyperbolic if and only if $H[R]$ is the unknot, a torus knot, or a cable of a torus knot. Then Theorem~\ref{nonhyperbolicity from a proper power} completes the proof.
\end{proof}

Since the existence of a proper power curve indicates the nonhyperbolicity of P/P and P/SF knots, for a given R-R diagram of a curve $R$ in the boundary
of a genus two handlebody, it is useful to determine whether or not there is a proper power curve $\beta$ disjoint from $R$.

Based on the following minor generalization of a result of Cohen, Metzler, and Zimmerman, we are able to classify all possible R-R diagrams of a proper power curve $\beta$ for any complete set of cutting disks of a genus two handlebody $H$.

\begin{thm}\cite{CMZ81}
\label{recognizing primitives and proper powers}
Suppose a cyclic conjugate of
\[W = A^{n_1}B^{m_1} \dots A^{n_l}B^{m_l}\]
 is a member of a basis of $F(A,B)$ or a proper power of a member of a basis of $F(A,B)$, where $l \geq 1$ and each indicated exponent is nonzero. Then, after perhaps replacing $A$ by $A^{-1}$ or $B$ by $B^{-1}$, there exists $e > 0$ such that:
\[
n_1 = \dots = n_l = 1,
\quad
\text{and}
\quad
\{m_1, \dots ,m_l\} \subseteq \{e, e+1\},
\]
or
\[
\{n_1, \dots ,n_l\} \subseteq \{e, e+1\},
\quad
\text{and}
\quad
m_1 = \dots = m_l = 1.
\]
\end{thm}

\begin{thm}\label{main theorem4}
Suppose $H$ is a genus two handlebody with a complete set of cutting disks $\{D_A, D_B\}$ with $\pi_1(H)=F(A,B)$, where
the generators $A$ and $B$ are dual to $D_A$ and $D_B$ respectively.
If $\beta$ is a proper power curve in $H$ which has only essential intersections with $D_A$ and $D_B$,
then $\beta$ has one of the following R-R diagrams with respect to the complete set of cutting disks $\{D_A, D_B\}$
up to the homeomorphisms of $H$ inducing the automorphisms
exchanging $A$ and $B$, and replacing $A^{-1}$ by $A$ or $B^{-1}$ by $B$:
\begin{enumerate}
\item \emph{Type I:} $\beta$ has an R-R diagram with a $0$-connection in at least one of the handles.
\item \emph{Type II:} $\beta$ has an R-R diagram of the form shown in Figure~\emph{\ref{PPower8}} with $s>1$.
\item \emph{Type III:} $\beta$ has an R-R diagram of the form shown in Figure~\emph{\ref{PPower10}} with $a,b>0$ and $s>0$.
\item \emph{Type IV:} $\beta$ has an R-R diagram of the form shown in Figure~\emph{\ref{PPower10-1}} with $a,b,c>0$.
\item \emph{Type V:} $\beta$ has an R-R diagram of the form shown in Figure~\emph{\ref{PPower11-1}} with $a,b,c,d>0$.
\end{enumerate}
\end{thm}
\begin{proof}
This is Theorem 1.3 of \cite{K20c}.
\end{proof}

\begin{figure}[t]
\centering
\includegraphics[width = 0.7\textwidth]{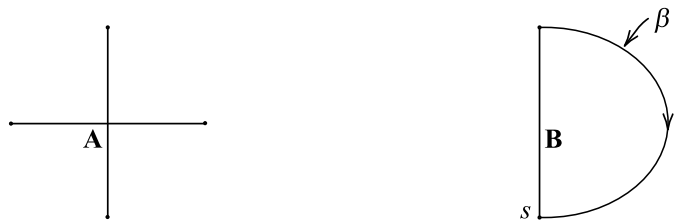}
\caption{Type II of proper power curves $\beta$, where $s>1$.}
\label{PPower8}
\end{figure}

\begin{figure}[t]
\centering
\includegraphics[width = 0.6\textwidth]{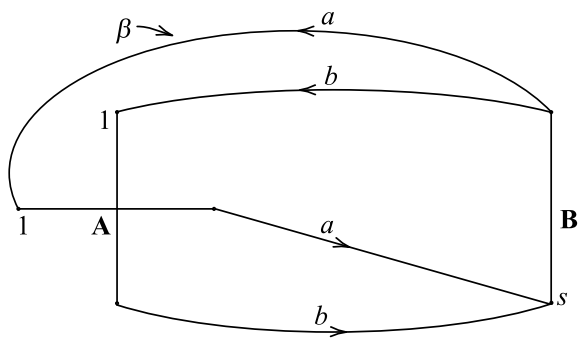}
\caption{Type III of proper power curves $\beta$: $[\beta]=(AB^s)^{a+b}$, where $s>0$ and $a,b>0$.}
\label{PPower10}
\end{figure}

\begin{figure}[t]
\centering
\includegraphics[width = 0.7\textwidth]{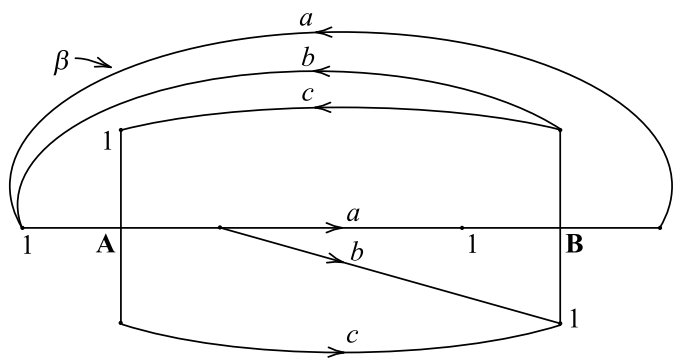}
\caption{Type IV of proper power curves $\beta$: $[\beta]=(AB)^{a+b+c}$, where $a,b,c>0$.}
\label{PPower10-1}
\end{figure}

\begin{figure}[t]
\centering
\includegraphics[width = 0.7\textwidth]{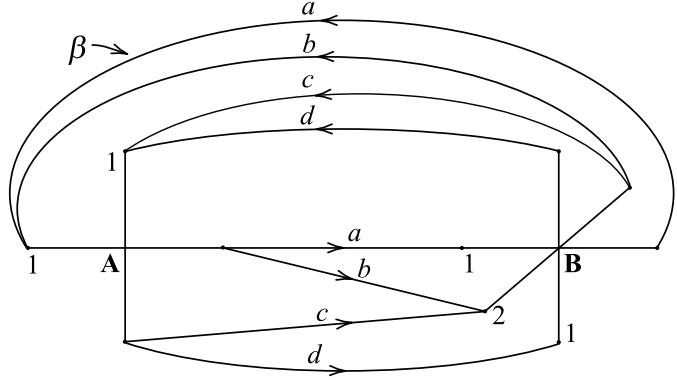}
\caption{Type V of proper power curves $\beta$, where $a,b,c,d>0$.}
\label{PPower11-1}
\end{figure}

Theorem~\ref{main theorem4} yields some properties of a proper power curve and
together with Theorem~\ref{nonhyperbolicity from a proper power} provides some necessary condition for $H[R]$ being nonhyperbolic.

\begin{prop}\label{a proper power is disconnected and has a cut-vertex}
Suppose $\beta$ is a proper power curve in the boundary of a genus two handlebody $H$. Then the Heegaard diagram of $\beta$
for any complete set of cutting disks of $H$ either is not connected or has a cut-vertex.
\end{prop}
\begin{proof}
It is easy to see that the Heegaard diagrams of $\beta$ underlying the R-R diagrams of $\beta$ in Figures~\ref{PPower8}, \ref{PPower10}, \ref{PPower10-1}, and \ref{PPower11-1}
either are not connected or have a cut-vertex. Therefore we assume that an R-R diagram of $\beta$ has a $0$-connection on one handle.
Then $\beta$ has the corresponding Heegaard diagram of $\beta$ as illustrated in Figure~\ref{DPCFig8a-3}b. Theorem~\ref{recognizing primitives and proper powers}
forces at least one $a$ and $b$ be $0$, in which case the Heegaard diagram either is not connected or has a cut-vertex.
\end{proof}

\begin{prop}\label{necessary condition for pp curves}
Suppose $\beta$ is a proper power curve in the boundary of a genus two handlebody $H$ whose R-R diagram
of $\beta$ has $1$-connections but has no $0$-connections. Then $\beta$ satisfies the following:
\begin{enumerate}
\item The R-R diagram of $\beta$ has two bands of $1$-connections on one handle; and
\item $\beta$ does not appear with exponents $e$ and $e+1$ with $e>1$ in $\pi_1(H)$.
\end{enumerate}
\end{prop}
\begin{proof}
Since $\beta$ has an R-R diagram with $1$-connections but no $0$-connections, by Theorem~\ref{main theorem4}, it has the form in Figure~\ref{PPower10}, \ref{PPower10-1}, or \ref{PPower11-1}. Then it follows from Figures~\ref{PPower10}, \ref{PPower10-1}, and \ref{PPower11-1} that the R-R diagram of $\beta$ has two bands of 1-connections on one handle, and $\beta$ does not appear with exponents $e$ and $e+1$ with $e>1$ in $\pi_1(H)$, as desired.
\end{proof}

\begin{prop}\label{sufficient condition for hyperbolicity}
Suppose $H$ is a genus two handlebody with a complete set of cutting disks $\{D_A, D_B\}$ with $\pi_1(H)=F(A,B)$, where
the generators $A$ and $B$ are dual to $D_A$ and $D_B$ respectively, and
suppose $R$ is a simple closed curve on $\partial H$ such that $H[R]$ embeds in $S^3$.

Suppose $R$ has an R-R diagram with respect to $\{D_A, D_B\}$ which has at least two bands of connections in each handle.
Let $q$ and $r$ be the maximal labels of bands of connections of $R$ in the $A$- and $B$-handles respectively. If $R$
satisfies one of the following, then $H[R]$ is not the exterior of the unknot, a torus knot, or a cable knot.
\begin{enumerate}
\item $|q|, |r|>2$.
\item one of $|q|$ and $|r|$ is $2$, say $|r|=2$, and $|q|>2$, and $R$ has a subarc representing word $A^eB^rA^f$ with $e\neq f$ in $\pi_1(H)$.
\item $|q|=|r|=2$ and $R$ has two subarcs representing words $A^eB^rA^f$ and $B^gA^qB^h$ with $e\neq f$ and $g\neq h$ in $\pi_1(H)$.
\end{enumerate}
\end{prop}
\begin{proof}
Suppose for contradiction that $H[R]$ is the exterior of the unknot, a torus knot, or a cable knot. Then by Theorem~\ref{nonhyperbolicity from a proper power}, there exists a proper power curve $\beta$ in $H$ disjoint from $R$. Now we add the curve $\beta$ in the R-R diagram of $R$. By the classification of proper power curves in Theorem~\ref{main theorem4}, $\beta$ has an R-R diagram of one of the six types I-VI with respect to $\{D_A, D_B\}$. However the types I and II are impossible because the R-R diagram of $R$ has at least two bands of connections with the absolute value of the maximal label of bands of connection greater than or equal to $2$ in each handle.

For the types III, IV, and V of the R-R diagram of $\beta$, we can observe from Figures~\ref{PPower10}, \ref{PPower10-1}, and \ref{PPower11-1} that the possible maximal label of band of connection of the R-R diagram of $R$ in the $A$-handle would be $2$ up to sign and furthermore if $R$ has a subarc representing $B^eA^2B^f$ in $\pi_1(H)$, then $e=f$. This contradicts all the conditions (1), (2), and (3) in the hypothesis.
\end{proof}

\section{Finding $R$ from a P/P or P/SF knot $k$ and a meridian $M$ of $k$}\label{Finding $R$ from a P/P or P/SF knot $k$ and a meridian $M$ of $k$}

Suppose, as usual, $\alpha$ is P/P or P/SF in a genus two Heegaard splitting $(\Sigma; H, H')$ of $S^3$ assuming that $\alpha$ is primitive in $H'$ in both cases. Let $k$ be a knot in $S^3$ represented by $\alpha$. Then there is a unique cutting disk $D_R$ in $H'$ such that the boundary $R$ of $D_R$ is disjoint from $\alpha$ and $H[R]$ is homeomorphic to the exterior of $k$ in $S^3$. Suppose $M$ is the meridian of $H[R]$ such that $M\cap \alpha$
is a single point of transverse intersection. Let $F$ be a regular neighborhood of $M\cup \alpha$ in $\partial H$,
which is a once-punctured torus in $\partial H$, and let $\Gamma=\partial F$.

We suppose $\{D_A, D_B\}$ is a complete set of cutting disks of $H$ such that the graph $G(\Gamma)$ of the Heegaard diagram of $\Gamma$
with respect to $\{D_A, D_B\}$ is robust. In general, the graph $G(C)$ of the Heegaard diagram of a simple closed curve $C$ is defined to be \textit{robust}
if deleting any pair of edges, say $e, e'$ from $G(C)$, such that $\{e, e'\}$ is invariant under the hyperelliptic involution,
leaves a graph $G'$, which is connected and has no cut-vertex. As an example, if $C$ has a graph of the form in Figure~\ref{DPCFig8a-3}a with $a,b,c,d>0$,
then $G(C)$ is robust. While, if one of $c$ and $d$ is $0$ and $a=b=1$, then $G(C)$ is not robust. Note that if $C$ is
a separating curve such as $\Gamma$ above and has a graph of the form in Figure~\ref{DPCFig8a-3}a, then the weights $a$ and $b$ are odd, and $c$ and $d$ are even.

Let $F'=\overline{\partial H-F}$. Then the simple closed curve $R$ lies completely in $F'$.

\begin{lem}\label{Robust lemma with two bands}
$(\partial D_A\cup \partial D_B)\cap F$ and $(\partial D_A\cup \partial D_B)\cap F'$ consist of at lease two or three bands of connections in
$F$ and $F'$ respectively.
\end{lem}

\begin{figure}[t]
\includegraphics[width = 0.8\textwidth]{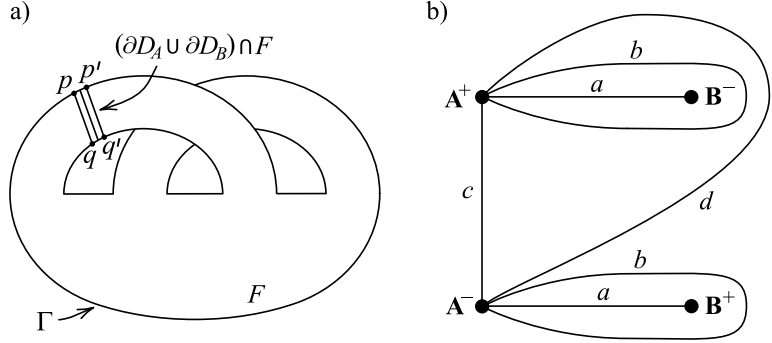}\caption{Only one band of connections of $(\partial D_A\cup \partial D_B)\cap F$ in $F$
and the Heegaard diagram of $\Gamma$.}\label{robust-1}
\end{figure}

\begin{proof}
Suppose $(\partial D_A\cup \partial D_B)\cap F$ consists of only one band of connections in $F$. Then it appears as in Figure~\ref{robust-1}a.
There exist two subarcs of $\Gamma$ each of which has both endpoints at one cutting disk. Figure~\ref{robust-1}a
shows such two subarcs one of which has two endpoints $p$ and $q$ and the other $p'$ and $q'$. It follows that
$G(\Gamma)$ is of the form in Figure~\ref{robust-1}b with $b>0$, which is disconnected or has a cut-vertex.
This is a contradiction to the robustness of $G(\Gamma)$. Similarly for $F'$.
\end{proof}

Let $B_1, B_2$, and $B_3$ denote the three bands, and let $P_1, P_2,$ and
$P_3$ denote paths in $F$ transverse to $B_1, B_2$, and $B_3$ respectively.
If there are two bands in $F$, then we assume that $B_3=\varnothing$.
See Figure~\ref{robust-2} when there are two or three bands.

\begin{figure}[t]
\includegraphics[width = 0.8\textwidth]{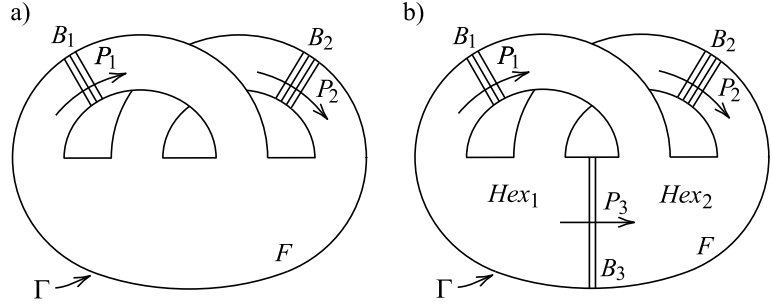}\caption{The three bands $B_1, B_2$, and $B_3$ of $(\partial D_A\cup \partial D_B)\cap F$
and the paths $P_1, P_2,$ and $P_3$ in $F$ transverse to $B_1, B_2$, and $B_3$ respectively. }\label{robust-2}
\end{figure}

When there are two bands, we let $\overline{P}_1$ and $\overline{P}_2$ be
simple closed curves disjoint from $B_2$ and $B_1$ respectively.

When there are three bands, $B_1, B_2$, and $B_3$ cut $F$ into rectangles and two nonrectangular faces $Hex_1$ and $Hex_2$.
In this case, suppose $P_1, P_2,$ and $P_3$ are oriented paths from $Hex_1$ to $Hex_2$ as illustrated in Figure~\ref{robust-2}b.
We let $P_1P_2^{-1}, P_1P_3^{-1}$, and $P_2P_3^{-1}$ be oriented simple closed curves disjoint from $B_3, B_2,$ and $B_1$, which are
induced by connecting the two oriented paths $P_1$ and $P_2^{-1}$, $P_1$ and $P_3^{-1}$, and $P_2$ and $P_3^{-1}$ respectively.
Similarly, we can define $B'_1, B'_2, B'_3, P'_1, P'_2, P'_3$, etc for $F'$.

Now we suppose that $H[R]$ is the exterior of the unknot, a torus knot, or cable knot in $S^3$.
Theorem~\ref{nonhyperbolicity from a proper power} indicates that there exists a proper power curve $\beta$ in $\partial H$, disjoint from $R$.
Since $R$ lies in $F'$, either $\beta$ lies completely in $F$ or $\beta$ has an essential intersection with $\Gamma$.

\begin{prop}\label{Robust lemma in F}
If $\beta$ lies completely in $F$, then $\beta$ must be disjoint from one of the bands of $(\partial D_A\cup \partial D_B)\cap F$.
Therefore, if there are two bands $B_1$ and $B_2$ of $(\partial D_A\cup \partial D_B)\cap F$,
then $\beta$ is either $\overline{P}_1$ or $\overline{P}_2$. If there are three bands $B_1, B_2$, and $B_3$,
then $\beta$ is either $P_1P_2^{-1}, P_1P_3^{-1},$ or $P_2P_3^{-1}$.
\end{prop}

\begin{figure}[t]
\includegraphics[width = 0.8\textwidth]{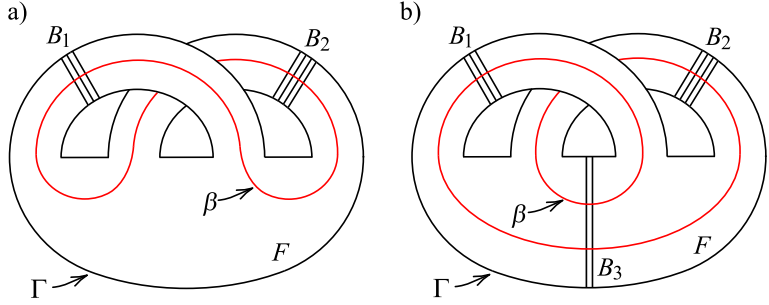}\caption{A proper power $\beta$ which
intersects each of the bands of $(\partial D_A\cup \partial D_B)\cap F$.}\label{robust-3}
\end{figure}

\begin{proof}
Suppose that $\beta$ intersects each of the bands of $(\partial D_A\cup \partial D_B)\cap F$. Then the graph $G(\beta)$ of
the Heegaard diagram of $\beta$ with respect to $\{D_A, D_B\}$ omits at most two edges of $G(\Gamma)$,
which are invariant under the hyperelliptic involution. See Figure~\ref{robust-3}, where $\beta$ intersects each of the bands of $(\partial D_A\cup \partial D_B)\cap F$.
Since $G(\Gamma)$ is robust, $G(\beta)$ is connected and has no cut-vertex. This is contradiction to Proposition~\ref{a proper power is disconnected and has a cut-vertex}.
\end{proof}

\begin{lem}\label{Robust lemma in F'}
Suppose that $\beta$ has essential intersections with $\Gamma$. Let $\delta$ be a connection in $\beta\cap F'$, and
suppose that $\delta$ has been isotoped, keeping its endpoints on $\Gamma$ so $\delta$ has only essential
intersections with the bands of $(\partial D_A\cup \partial D_B)\cap F'$. Then $\delta$ must be disjoint from one of the bands of $(\partial D_A\cup \partial D_B)\cap F'$.
\end{lem}

\begin{figure}[t]
\includegraphics[width = 0.8\textwidth]{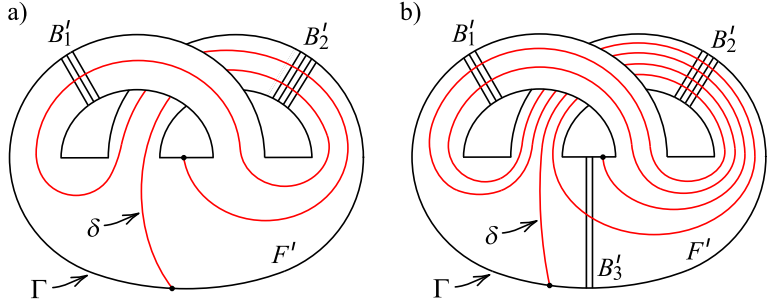}\caption{$\delta$ intersects each of the bands of $(\partial D_A\cup \partial D_B)\cap F'$ essentially.}\label{robust-4}
\end{figure}

\begin{proof}
The proof is similar to that of Proposition~\ref{Robust lemma in F}. If $\delta$ intersects each of the bands of $(\partial D_A\cup \partial D_B)\cap F'$,
then the graph $G(\delta)$ of the Heegaard diagram of $\delta$ with respect to $\{D_A, D_B\}$ omits at most two edges of $G(\Gamma)$.
See Figure~\ref{robust-4}. Since $G(\Gamma)$ is robust, $G(\delta)$ is connected and has no cut-vertex
and so is $G(\beta)$, a contradiction.
\end{proof}

Let $I=\{1,2\}(\{1,2,3\}$, resp.) when there are two(three, resp.) bands of $(\partial D_A\cup \partial D_B)\cap F'$.
For $i\in I$, let $|P'_i|$ denote the number of appearances of ${P'_i}^{\pm1}$ in $R$. Note that $R$ can be expressed
in terms of the oriented paths $P'_i$'s as we define $P'_1{P'_2}^{-1}, P'_1{P'_3}^{-1}$, and $P'_2{P'_3}^{-1}$ as simple closed curves in $F'$. The following proposition
describes all possible candidates for $R$ when $\beta$ has an essential intersection with $\Gamma$.

\begin{prop}\label{Possible R in F' from Robust lemma}
Suppose that $\beta$ has an essential intersection with $\Gamma$.
Then there exists $i_0\in I$ such that $|P'_{i_0}|=1$.

Therefore if there are two bands of $(\partial D_A\cup \partial D_B)\cap F'$, then the possible $R$
is of the form $\overline{P'}_{i_0}$ ${\overline{P'}_j}^{m}$ for $j\in\{1,2\}$ with $j\neq i_0$ and $m\in\mathbb{Z}$.

If there are three bands of $(\partial D_A\cup \partial D_B)\cap F'$,
the possible $R$ is of the form $P'_{i_0}{P'_j}^{-1}(P'_k{P'_j}^{-1})^m$ for distinct $j, k\in\{1,2,3\}$ with $j, k\neq i_0$ and $m\in\mathbb{Z}$.
\end{prop}

\begin{proof}
Suppose that $|P'_{i}|>1$ for every $i\in I$. Then since any connection $\delta$ of $\beta$ in $F'$ is disjoint from $R$,
it follows that $\delta$ intersects each of $B_i$, $i\in I$, in $F'$, which is a contradiction to Lemma~\ref{Robust lemma in F'}.
\end{proof}

\section{Application on some P/SF knots}\label{Applications on some P/SF knots}
In this section, we give applications of the nonhyperbolic conditions for P/P or P/SF knots demonstrated in the previous sections to some P/SF knots.
Consider an abstract R-R diagram of a pair of simple closed curves $(\alpha, M)$ in the boundary of a genus two handlebody $H$ shown in Figure~\ref{kisthyp1} with $p>0$ and $|J|>1$, where $M$ is a meridian of $H[R]$. In other words, $\{M, R\}$ bound a complete set of cutting disks of $H'$ such that $(\Sigma; H, H')$ is a genus two Heegaard splitting of $S^3$ and $R$ is disjoint from $\alpha$. By Theorem 5.1 of \cite{K20c}, $H[\alpha]$ is Seifert-fibred over the M\"{o}bius band with one exceptional fiber of index $|J|$ and thus $\alpha$ is Seifert in $H$. In addition, since $M\cap \alpha$
is a single point of transverse intersection, $\alpha$ is primitive in $H$. Therefore $\alpha$ is primitive/Seifert.

\begin{lem}\label{hyperbolicity in KIST}
All of the knots represented by $\alpha$ in Figure~\emph{\ref{kisthyp1}} are hyperbolic.
\end{lem}

\begin{proof}
Suppose the knot represented by $\alpha$ is not hyperbolic. Then by Theorem~\ref{nonhyperbolicity from a proper power}, there exists a proper power curve $\beta$ in $\partial H$ that is disjoint from $R$.

\begin{figure}[t]
\includegraphics[width = 0.6\textwidth]{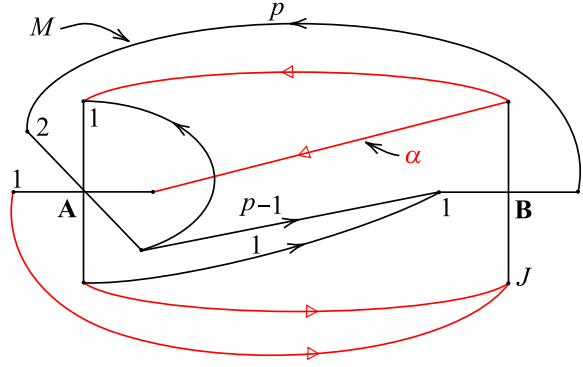}\caption{An abstract R-R diagram of a KIST pair $(\alpha, M)$, where $p>0$ and $|J|>1$.}\label{kisthyp1}
\end{figure}

Let $F$ be a regular neighborhood of $\alpha\cup M$ in $\partial H$. Let $\Gamma=\partial F$ and $F'=\overline{\partial H-F}$. Then $F$ and $F'$ are once-punctured tori with $\Gamma=F\cap F'$. Note that $R$ lies completely in $F'$. Since $\alpha=A^{-1}B^JAB^J$ and $M=A(BA^2)^p$ in $\pi_1(H)$ when reading from their intersection point, algebraically $\Gamma=M\alpha M^{-1}\alpha^{-1}$ is
$$\Gamma=(A^2B)^pAB^JAB^J(A^{-2}B^{-1})^pA^{-1}B^{-J}A^{-1}B^{-J}.$$
The R-R diagram of $\Gamma$ when $p=2$ appears as in Figure~\ref{kisthyp2}, where some faces which belong to $F$ or $F'$ are labelled by the letter $F$ or $F'$ respectively. Then $\Gamma$ has a Heegaard diagram $\mathbb{D}_\Gamma$ underlying the R-R diagram as illustrated in Figure~\ref{kisthyp3} depending on the sign of $J$.

\begin{figure}[t]
\includegraphics[width = 0.6\textwidth]{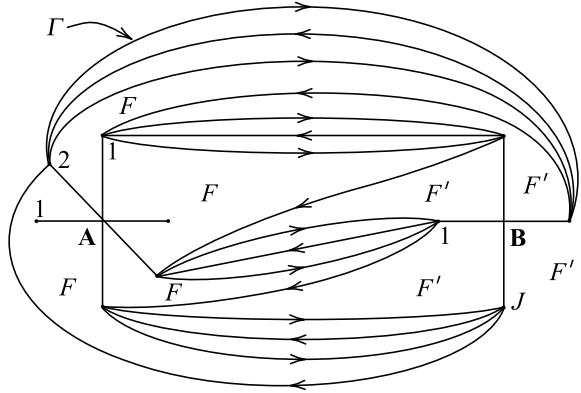}\caption{R-R diagram of $\Gamma$ when $p=2$.}\label{kisthyp2}
\end{figure}

\begin{figure}[t]
\includegraphics[width = 0.75\textwidth]{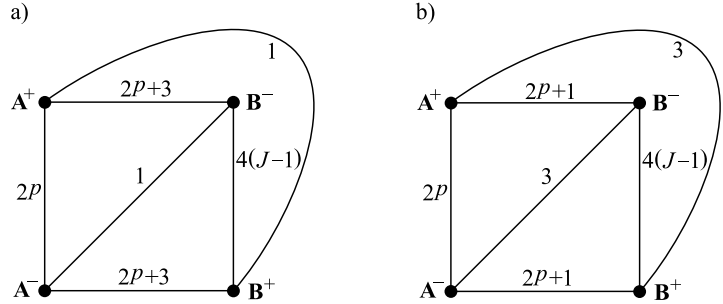}\caption{The Heegaard diagram $\mathbb{D}_\Gamma$ of $\Gamma$ when $J>1$ in a) and $J<1$ in b) underlying the R-R diagram of $\Gamma$.}\label{kisthyp3}
\end{figure}

It follows that the Heegaard diagram $\mathbb{D}_\Gamma$ is robust, and the arcs of $(\partial D_A\cup \partial D_B)\cap F$($(\partial D_A\cup \partial D_B)\cap F'$, resp.) cut $F$($F'$, resp.) into a number of faces, each of which is a rectangle, except for a pair of hexagonal faces. Therefore $(\partial D_A\cup \partial D_B)\cap F$ and $\partial D_A\cup \partial D_B)\cap F'$ consist of three bands of connections in
$F$ and $F'$ respectively since there are two hexagonal faces in each of $F$ and $F'$.
Let $Hex_1$ and $Hex_2$($Hex'_1$ and $Hex'_2$, resp.) be the two hexagonal faces in $F$($F'$, resp.). They can be obtained in the R-R diagram of $\Gamma$ by isotoping the oriented curves entering into or coming out from the maximally labelled connection in each handle to the oriented curves passing through the other two connections with smaller labels. In the R-R diagram of $\Gamma$, $2$ and $J$ are the maximal labels in the $A$- and $B$-handles respectively.
Note that the isotopies in the $B$-handle depend on the sign of $J$.  Figures~\ref{kisthyp4}a and \ref{kisthyp4}b show the isotopies of the curves passing through the 2-connections and the $J$-connections into the dotted curves passing through the other two connections when $J>1$ and $J<-1$ respectively. With these isotopies performed, we can obtain the hexagonal faces $Hex_1$ and $Hex_2$ in $F$ and $Hex'_1$ and $Hex'_2$ in $F'$ as shown in Figure~\ref{kisthyp4}.

\begin{figure}[t]
\includegraphics[width = 0.65\textwidth]{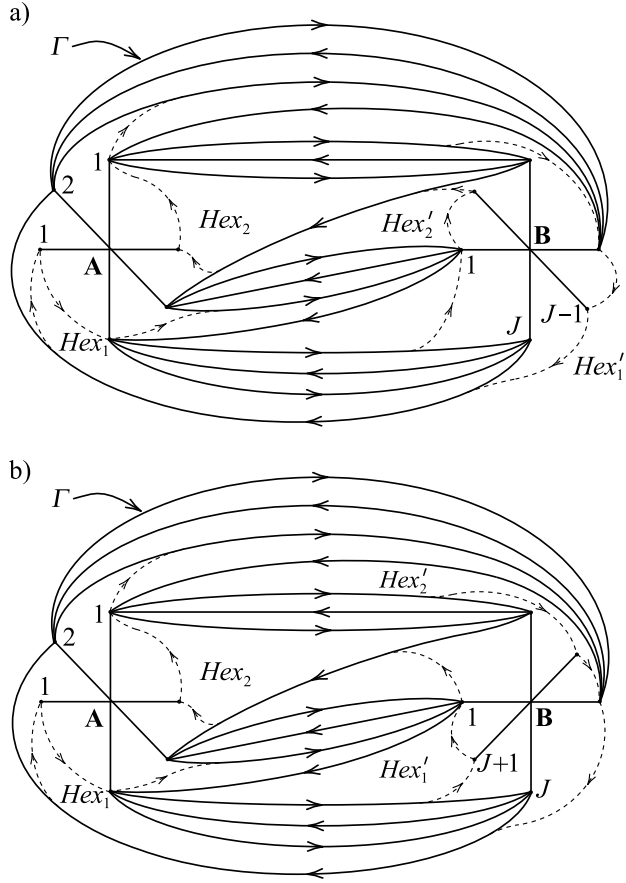}\caption{Hexagonal faces $Hex_1$ and $Hex_2$ in $F$, and $Hex'_1$ and $Hex'_2$ in $F'$ when a) $J>1$ and b) $J<-1$.}\label{kisthyp4}
\end{figure}

As discussed in Section~\ref{Finding $R$ from a P/P or P/SF knot $k$ and a meridian $M$ of $k$}, we can obtain three oriented paths $P_1, P_2,$ and $P_3$ from $Hex_1$ to $Hex_2$ in $F$, whose first homotopies in $\pi_1(H)$ are
$$P_1=A,\hspace{0.05cm} P_2=B^JAB^J, \hspace{0.15cm}\text{and} \hspace{0.16cm} P_3=A^{-1}(B^{-1}A^{-2})^{p-1}B^{-1}A^{-1}.$$
Similarly there are three oriented paths $P'_1, P'_2,$ and $P'_3$ from $Hex'_1$ to $Hex'_2$ in $F'$, whose first homotopies in $\pi_1(H)$ are when $J>1$
$$P'_1=B^{J-1},\hspace{0.05cm} P'_2=B^{-1}A^{-1}B^{-J}A^{-1}B^{-1}, \hspace{0.15cm}\text{and} \hspace{0.16cm} P'_3=(A^2B)^{p-1}A^2,$$
\noindent and when $J<-1$
$$P'_1=B^{J+1},\hspace{0.05cm} P'_2=A^{-1}B^{-J}A^{-1}, \hspace{0.15cm}\text{and} \hspace{0.16cm} P'_3=B(A^2B)^p.$$

Now we consider the proper power curve $\beta$.

\begin{claim}\label{claim1}
$\beta$ has an essential intersection with $\Gamma$.
\end{claim}
\begin{proof}
Suppose $\beta$ has no essential intersections with $\Gamma$, i.e., $\beta\cap\Gamma=\varnothing$. Since $R\subseteq F'$, $\beta$ must lie completely in $F$. Since the Heegaard diagram $\mathbb{D}_\Gamma$ of $\Gamma$ is robust, by Proposition~\ref{Robust lemma in F}, $\beta$ must be one of the curves $P_1P_2^{-1}, P_1P_3^{-1}$, and $P_2P_3^{-1}$. It is easy to check that $P_1P_2^{-1}=\alpha^{-1}, P_1P_3^{-1}=M$. Therefore $\beta=P_2P_3^{-1}$ whose homotopy is $\beta=B^JAB^JAB(A^2B)^{p-1}A$. If $p>1$, then $A^2$ and $B^J$ appear, implying that by Proposition~\ref{necessary condition for pp curves} $\beta$ cannot be a proper power curve. Therefore $p=1$ and $\beta=B^JAB^JABA$. The automorphism $A\mapsto B^{-J}A$ sends $\beta$ to $\beta=A^3B^{1-J}$, which is not a proper power curve.
\end{proof}

By Claim~\ref{claim1} $\beta$ has an essential intersection with $\Gamma$. Since $\Gamma$ is robust, by Proposition~\ref{Possible R in F' from Robust lemma} the possible curves for $R$ are the following simple closed curves in $F'$ up to orientations:
$$P'_{1}{P'_2}^{-1}(P'_3{P'_2}^{-1})^m,\hspace{0.05cm} P'_{2}{P'_1}^{-1}(P'_3{P'_1}^{-1})^m, \hspace{0.15cm}\text{and} \hspace{0.16cm} P'_{3}{P'_2}^{-1}(P'_1{P'_2}^{-1})^m$$
where $m\in\mathbb{Z}$, such that when $J>1$,

\begin{enumerate}
\item $P'_{1}{P'_2}^{-1}(P'_3{P'_2}^{-1})^m=B^JAB^JAB((A^2B)^{p-1}A^2BAB^JAB)^m$,
\item $P'_{2}{P'_1}^{-1}(P'_3{P'_1}^{-1})^m=B^{-1}A^{-1}B^{-J}A^{-1}
    B^{-J}((A^2B)^{p-1}A^2B^{-J+1})^m$,
\item $P'_{3}{P'_2}^{-1}(P'_1{P'_2}^{-1})^m=(A^2B)^{p-1}A^2BAB^JAB(B^JAB^JAB)^m$,
\end{enumerate}

\noindent and when $J<-1$,

\begin{enumerate}
\item[(4)] $P'_{1}{P'_2}^{-1}(P'_3{P'_2}^{-1})^m=B^{J+1}AB^JA(B(A^2B)^pAB^JA)^m$,
\item[(5)] $P'_{2}{P'_1}^{-1}(P'_3{P'_1}^{-1})^m=A^{-1}B^{-J}A^{-1}B^{-J-1}
    (B(A^2B)^pB^{-J-1})^m$,
\item[(6)] $P'_{3}{P'_2}^{-1}(P'_1{P'_2}^{-1})^m=B(A^2B)^pAB^JA(B^{J+1}AB^JA)^m$.
\end{enumerate}

Now we compute the determinant $\Delta=|[M] [R]|$, where $[M], [R]\in H_1(H)$. $\Delta$ must be unimodular because $H[M, R]\cong S^3$.
If $J>1$ and $R$ is the curve in the first case (1), i.e., $R=P'_{1}{P'_2}^{-1}(P'_3{P'_2}^{-1})^m=B^JAB^JAB((A^2B)^{p-1}A^2BAB^JAB)^m$, then
\[
\begin{split}
|[M] [R]| &= \left|\begin{array}{cc}
2p+1 &  (2p+2)m+2\\
p  & (p+J+1)m+2J+1\\
\end{array} \right|
=\left|\begin{array}{cc}
2p+1 &  m+2\\
p  & (J+1)(m+2)-1\\
\end{array} \right|\\
&=((2p+1)J+p+1)(m+2)-(2p+1).
\end{split}
\]
The second matrix is obtained from the first matrix by adding $-m$ multiple of the first column to the second column. Similarly we can compute the determinants in the other five cases (2)--(6).  All of the determinants $\Delta$ are as follows, where $\Delta_i$ denotes the determinant in the case $(i)$ with $i=1, \ldots, 6$.

\begin{enumerate}
\item $\Delta_1=((2p+1)J+p+1)(m+2)-(2p+1)$,
\item $\Delta_2=(-(2p+1)J+p)(m+2)-(2p+1)$,
\item $\Delta_3=((2p+1)J-p)(2m+1)+(2p+1)(m+1)$,
\item $\Delta_4=((2p+1)(J+1)-p)(m+2)-(2p+1)$,
\item $\Delta_5=((2p+1)J-p)(m+2)+(2p+1)$,
\item $\Delta_6=((2p+1)(J+1)-p)(2m+1)-(2p+1)(m+1)$.
\end{enumerate}

It is easy to see that for all the values of $p, J,$ and $m$ with $p>0$, $|J|>1$, and $m\in\mathbb{Z}$, $\Delta_i\neq \pm1$ for $i=1, \ldots, 6$, only except
for $\Delta_4$ with the set of the parameters $(p, J, m)=(1, -2, -3)$, in which case $\Delta_4=1$.

In order to complete the proof, we use Proposition~\ref{sufficient condition for hyperbolicity}. Since we have assumed that the knot represented by $\alpha$ is not hyperbolic, $H[R]$ is the exterior of the unknot, a torus knot, or a cable knot. Therefore the curve $R$ can satisfy none of the conditions (1), (2), and (3) in Proposition~\ref{sufficient condition for hyperbolicity}, provided that $R$ has at least two bands of connections in each handle. With $(p, J, m)=(1, -2, -3)$ inserted in the case (4), $R$ is
\[
\begin{split}
R &=B^{-1}AB^{-2}A(BA^2BAB^{-2}A)^{-3}\\
&=B^{-3}A^{-2}(B^{-1}A^{-1}B^2A^{-1}B^{-1}A^{-2})^2.
\end{split}
\]
Since $A^{-1}, A^{-2}, B^{-1}, B^{2}$ and $B^{-3}$ appear in $R$, $R$ has at least two bands of connections in both the $A$- and $B$-handles. Furthermore, $B^{-3}A^{-2}B^{-1}$ appears in $R$ and thus $R$ satisfies the condition (2) in Proposition~\ref{sufficient condition for hyperbolicity}, a contradiction. This completes the proof.
\end{proof}


\end{document}